\documentclass[a4paper,12pt]{amsart}
\usepackage{amsfonts,amsthm,amssymb,amsmath,amscd}
\usepackage[colorlinks=true,linkcolor=blue,citecolor=blue]{hyperref}

\usepackage{tikz-cd}

\usepackage{tabularx}
\usepackage{tcolorbox}
\usepackage[english]{babel}
\usepackage{amssymb,amsmath}
\usepackage{xcolor}

\newtheorem{Theo}{Theorem}[section]
\newtheorem{Prop}[Theo]{Proposition}
\newtheorem{Coro}[Theo]{Corollary}
\newtheorem{Lemm}[Theo]{Lemma}
\newtheorem{Defi}[Theo]{Definition}

\newtheorem{Ques}[Theo]{Question}
\newtheorem{Rema}[Theo]{Remark}

\DeclareMathOperator{\Maj}{Maj}
\DeclareMathOperator{\sign}{sign}

\def\N{\mathbb{ N}}
\def\R{\mathbb{ R}}

\begin{document}

\title{Bohr's phenomenon for functions on \\the Boolean cube}

\author[Defant]{Andreas Defant}
\address{\newline  Institut f\"{u}r Mathematik,\newline Carl von Ossietzky Universit\"at,\newline
26111 Oldenburg, Germany.
}
\noindent
\email{defant@mathematik.uni-oldenburg.de}

\author[Masty{\l}o]{Mieczys{\l}aw Masty{\l}o}
\address{\newline Faculty of Mathematics and Computer Sciences,\newline
Adam Mickiewicz University in Pozna\'n, \newline
Pozna\'n, Poland.
}
\noindent
\email{mastylo@amu.edu.pl}

\author[P\'erez]{Antonio P\'erez}
\address{\newline Departamento de Matem\'{a}ticas,\newline Universidad de Murcia,\newline
30100 Espinardo (Murcia), Spain.
}
\noindent
\email{antonio.perez7@um.es}

\maketitle

\begin{abstract}
\noindent
We study the asymptotic decay of the Fourier spectrum of real functions \mbox{$f\colon \{-1,1\}^N \rightarrow  \mathbb{R}$}
in the spirit of Bohr's phenomenon from complex analysis. Every such function admits a canonical representation through its
Fourier-Walsh expansion $f(x) = \sum_{S\subset \{1,\ldots,N\}}\widehat{f}(S) x^S \,,$ where $x^S = \prod_{k \in S} x_k$.
Given a~class $\mathcal{F}$  of functions on the Boolean cube $\{-1, 1\}^{N} $, the  Boolean radius of $\mathcal{F}$ is defined
to be the largest $\rho \geq 0$ such that $\sum_{S}{|\widehat{f}(S)| \rho^{|S|}} \leq \|f\|_{\infty}$  for every $f \in \mathcal{F}$.
We give the precise asymptotic behaviour of the Boolean radius of several natural subclasses of functions on finite Boolean cubes, as e.g. the
class of all real functions on $\{-1, 1\}^{N}$, the subclass made of all homogeneous functions or certain threshold functions.
Compared with the classical complex  situation subtle differences as well as striking parallels occur.
\end{abstract}

\vspace{10 mm}

\noindent
\renewcommand{\thefootnote}{\fnsymbol{footnote}}
 \footnotetext{The second named author was supported by the National Science Centre, Poland,
project no. 2015/17/B/ST1/00064. The research of the third author was partially done during a stay in Oldenburg (Germany) under the support of a PhD fellowship of ``La Caixa Foundation'', and of the projects of MINECO/FEDER (MTM2014-57838-C2-1-P) and Fundaci\'{o}n S\'{e}neca - Regi\'{o}n de Murcia (19368/PI/14).}.


\section{Introduction}
Recent years show a remarkable increase of articles on Bohr's phenomenon in  complex analysis. A~special feature is given
by the so-called Bohr radii of  holomorphic functions in several variables.

The main aim of this article is to start a parallel study for real functions $f\colon \{-1, 1\}^{N}\rightarrow \mathbb{R}$
on the Boolean cube $\{-1, 1\}^{N}$. The analysis of this type of maps is of fundamental interest in theoretical computer sciences, social choice, combinatorics, graph theory, etc. We refer to the
surveys of  \cite{RyanSurvey} and \cite{Wolf}, or the more extensive book \cite{BooleanRyan}.

Bohr's famous power series theorem from \cite{Bohr} states that  for each $d\in \mathbb{N}$ and all complex polynomials
$f(z)= \sum_{k=0}^d c_k z^k$, $z \in \mathbb{C}$  we have that
\begin{align*}
\sum_{k=0}^d |c_k| \frac{1}{3^k} \leq \|f\|_{\mathbb{T}}\,,
\end{align*}
where  $\|f\|_{\mathbb{T}}$ stands for  the supremum norm on the torus  $\mathbb{T}=\{ z\colon  |z|=1 \}$. Moreover, the radius $r=1/3$
is best possible. In terms of Fourier analysis this can be equivalently formulated as
\begin{equation} \label{power}
\sum_{k=0}^d |\widehat{f}(k)| \frac{1}{3^k} \leq \|f\|_{\mathbb{T}}\,,
\end{equation}
where $\widehat{f}(k)$ as usual denotes the $k$th Fourier coefficient of $f$ as a function on $\mathbb{T}$.
The concept of a~multidimensional Bohr radius was introduced by Boas and Khavinson in \cite{BoKha}: Given
$N \in \mathbb{N}$ the $N$th Bohr radius $K_N$ is the best  (i.e. largest) $0 < r < 1 $ such that for every polynomial $f$
in $N$ complex variables we have
\begin{align} \label{BoKa}
\sum_{\alpha \in \mathbb{N}_0^N}{|\widehat{f}(\alpha)| \,r^{\alpha}} \leq \|f\|_{\mathbb{T}^N}\,,
\end{align}
where now  the $\alpha$th Fourier coefficient of $f$ is given by
\[
\widehat{f}(\alpha) = \int_{\mathbb{T}^N} f(w) w^{-\alpha}\,dw
\]
($dw$ is the normalized Lebesgue measure on the $N$-dimensional torus $\mathbb{T}^N$
and $\|f\|_{\mathbb{T}^N}$ denotes the supremum norm of $f$ on $\mathbb{T}^N$).
Obviously, $K_{1} = 1/3$, but we note that for no other index  $N$ the precise value of $K_N$ is known.
On the other hand, it was shown in \cite{BoKha} through probabilistic methods that
\[
\limsup_{N\to \infty}\sqrt{\tfrac{N}{\log N}}\, K_N \,\leq 1\,,
\]
and in \cite{Annals} based on the hypercontractivity of the Bohnenblust-Hille inequality that
\[
\sqrt{\tfrac{N}{\log N}}\, K_N  \ge \frac{1}{\sqrt{2}} + o(1).
\]
Improving the techniques which lead to the preceding result the main result of \cite{Bohrradius} states that in fact
\begin{align} \label{best}
\lim_{N\rightarrow \infty}\sqrt{\tfrac{N}{\log N}}\, K_N =1\,.
\end{align}
Note that this equality is basically a~result on the asymptotic decay of Fourier spectrum
$(\widehat{f}(\alpha))_{\alpha \in \mathbb{N}_0^N}$ of  polynomials $f$ on the  compact abelian group
$\mathbb{T}^N$ in high dimensions $N$.

The main aim of this article is to study Bohr's phenomenon replacing the $N$-dimensional torus $\mathbb{T}^N$
by the  $N$-dimensional Boolean cube $\{-1, 1\}^{N}$ ($\{\pm 1\}^{N}$ for short). We can look at $\{\pm 1\}^{N}$
as a~compact abelian group endowed with the coordinatewise product and the  discrete
topology (the Haar measure is then given by the normalized counting measure). This way, for every function
$f:\{ \pm 1\}^{N} \rightarrow \mathbb{R}$ we have an integral or expectation given by
\[
\mathbb{E}\big[ f \big] := \frac{1}{2^{N}} \sum_{x \in \{ \pm 1\}^{N}}{f(x)}\,,
\]
and for every $1 \leq p <  \infty$ the $L_{p}$-norm of $f$ is given by
\begin{equation} \label{p-norm}
 \| f\|_{p} = \mathbb{E}\big[ |f(x)|^{p} \big]^{1/p}.
 \end{equation}
The dual group of $\{ \pm 1\}^{N}$ actually consists of the set of all Walsh functions \mbox{$\chi_{S}: \{ \pm1\}^{N} \rightarrow \{ \pm 1\}$} indexed on $S \subset [N]:=\{n \colon 1 \leq n \leq N\}$ and defined as $$\chi_S(x) = x^{S}:=\prod_{k \in S} x_k \quad\, \big( x^{\emptyset} := 1 \big).$$
This let us associate to every function $f$ as above its
Fourier-Walsh expansion
\begin{equation}\label{equa:FourierExpansionShort}
f(x) = \sum_{S \subset [N]}{\widehat{f}(S) \, x^S}, \quad\, x \in \{\pm 1\}^{N}\,,
\end{equation}
where the Fourier coefficients are given by $\widehat{f}(S) = \mathbb{E}\big[ f \cdot \chi_{S} \big]$. The degree of a~nonzero $f$ is given by
\[
\deg f = \max \{|S|\colon \widehat{f}(S)\neq 0\}\,.
\]
We moreover say that $f$ is $d$-homogeneous whether $\widehat{f}(S) = 0$
for all $S \subset [N]$ with  $|S| \neq d$, and simply homogeneous if it is $d$-homogeneous for some $0 \leq d \leq N$. For $f=0$ we agree
to consider it as a~$0$-homogeneous or degree-$0$ function, like the rest of constant functions.

We introduce now the main definition of the paper.

\begin{Defi}
The Boolean radius of a function $f\colon \{\pm 1\}^{N} \rightarrow \mathbb{R}$ is the positive real number $\rho = \rho(f)$ satisfying
\[
\sum_{S \subset [N]}{|\widehat{f}(S)| \rho^{|S|}} = \| f \|_{\infty}.
\]
\end{Defi}

In general, calculating the exact Boolean radius of a~given function may be very complicated,  as complicated as determining the Fourier coefficients. However it may be possible to give estimations when we know some feature of the function such as the number of variables on which it depends, the distribution of the Fourier spectrum (whether it is concentrated in low levels or in sets with the same cardinality) or the size of $\mathbb{E}[f]$ in relation to $\| f\|_{\infty}$. For that let us consider the following concept:
\begin{Defi}
Given a~class $\mathcal{F}$  of functions on the Boolean cube $\{\pm{1}\}^N$, the  Boolean radius of $\mathcal{F}$ is defined as
\begin{align*}
\rho(\mathcal{F})  = \sup{\Big\{ \rho \geq 0 \colon \sum_{S \subset [N]}{|\widehat{f}(S)| \rho^{|S|}} \leq \| f\|_{\infty} \, \mbox{ for every }f \in \mathcal{F} \Big\}}\,,
\end{align*}
or equivalently
\begin{align*}
\rho(\mathcal{F}) = \inf{\{ \rho(f) \colon f \in \mathcal{F} \}}\,.
\end{align*}
\end{Defi}

\noindent We are going to concentrate on the following five classes defined for $N \in \N$, $0 \leq d \leq N$
and $0\leq \delta \leq 1$:\\

\begin{enumerate}
\setlength\itemsep{0.8em}
\item[] $\mathcal{B}^{N}$ :=  all functions $f:\{ \pm 1\}^{N} \rightarrow \mathbb{R}$,
\item[] $\mathcal{B}^{N}_{=d}$  :=  all  $d$-homogeneous $f:\{ \pm 1\}^{N} \rightarrow \mathbb{R}$,
\item[] $\mathcal{B}_{\text{hom}}^{N}$ :=   all homogeneous $f:\{ \pm 1\}^{N} \rightarrow \mathbb{R}$,
\item[] $\mathcal{B}^{N}_{\leq d}$ :=  all $f:\{ \pm 1\}^{N} \rightarrow \mathbb{R}$ with   $\deg f \leq d$,
\item[] $\mathcal{B}^{N}_{\delta}$ :=  all  $f:\{ \pm 1\}^{N} \rightarrow \mathbb{R}$ satisfying $|\mathbb{E}[f]|
\leq (1 - \delta) \| f\|_{\infty}$.\\
\end{enumerate}

\noindent
In section \ref{sec:nVariables}, we prove that the Boolean radius of $\mathcal{B}^{N}$ is given by the formula
\[
\rho(\mathcal{B}^{N}) = 2^{1/N} -1 = \frac{\log 2}{N} \big(1 + o(1)\big).
\]
In section \ref{sec:homogeneous} we deal with the families of homogeneous functions $\mathcal{B}_{=m}^{N}$ and $\mathcal{B}_{\text{hom}}^{N}$. We first prove in Theorem \ref{Theo:BooleanRadiusmHomogeneous} that for each $1 \leq m \leq N$ there exists $C_{m}$ (independent of $N$) such that
\[
C_{m}^{-1}\, N^{\frac{1}{2m}} \binom{N}{m}^{-\frac{1}{2m}} \leq \rho(\mathcal{B}_{=m}^{N}) \leq C_{m}\,\, N^{\frac{1}{2m}} \binom{N}{m}^{-\frac{1}{2m}},
\]
where moreover $\lim_{m}{C_{m}} = 1$. This means that for $m$ large, the previous estimation is asymptotically optimal. The proof for the lower bound relies on a~recently proved (Bohnenblust-Hille type) inequality  \eqref{Theo:PolyBH} while the upper estimate is based on probabilistic arguments (see Lemma \ref{Lemm:SalemZygmund}). The fact that $C_{m}$ converges to one is the key to prove
\[
\rho(\mathcal{B}_{\text{hom}}^{N}) = \sqrt{\frac{\log{N}}{N}} (1 + o(1)).
\]
(see Theorem \ref{Theo:BooleanRadiusHomogeneous}). In section \ref{sec:degreed} we deal with the case of function of degree $d$. Here the estimation for the Boolean radius of $\mathcal{B}_{\leq d}^{N}$ is less accurate than the homogeneous case. We show in Theorem \ref{Theo:BooleanRadiusDegree} that for each $1 \leq d \leq N$ there are constants $c_{d}$ and $C_{d}$ (independent of $N$) such that
\[
c_{d} \frac{1}{N^{\frac{1}{2}}} \leq \rho(\mathcal{B}_{\leq d}^{N}) \leq C_{d} \frac{1}{N^{\frac{d-1}{2d}}}.
\]
The proof of the lower bound requires an inequality for functions on $\{ \pm 1\}^{N}$ which  might be of independent interest, namely that if $f:\{ \pm 1\}^{N} \rightarrow \mathbb{R}$ is a~function of degree $d$, then
\[
\Big(\sum_{S \neq \emptyset}{|\widehat{f}(S)|^{2}}\Big)^{\frac{1}{2}}\leq 2 e^{d} (1 - |\widehat{f}(\emptyset)|).
\]
The preparatory work for that  reminds on inequalities
of  F. Wiener and  Caratheo\-dory
for complex polynomials (see Corollaries \ref{Coro:CaratheodoryTypeIneq} and \ref{Coro:CaratheodoryTypeIneq2} and the primal Theorem~\ref{Theo:MAINsumlessthanone}).
Finally, the last section \ref{sec:biased} is aimed to estimate the Boolean radius of $\mathcal{B}_{\delta}^{N}$. The main result (see Theorem \ref{Theo:BooleanRadiusUnbiased}) shows that there is an absolute constant $C> 0$ such that for each $N \in \N$ and $\frac{1}{2^{N}} \leq \delta \leq 1$
\[
C^{-1} \frac{1}{\sqrt{N} \sqrt{\log{(2/\delta)}}} \leq \rho(\mathcal{B}^{N}_{\delta}) \leq C \frac{1}{\sqrt{N} \sqrt{\log{(2/\delta)}}}.
\]
 Note that $\delta=\frac{1}{2^{N}}$ is the smallest ``sensitive'' value since for all $0\leq \delta \leq \frac{1}{2^{N}}$ we consequently have that
 \[ \Omega\Big(\frac{1}{N}\Big) = \rho(\mathcal{B}^{N}) \leq \rho(\mathcal{B}_{\delta}^{N}) \leq \rho(\mathcal{B}_{1/2^{N}}^{N}) = O\Big(\frac{1}{N}\Big). \]
The proof of the lower estimation relies on the
 well-known hypercontractivity  inequalities, while for the upper estimate we have to study the behaviour of the Boolean radius of the family of (threshold type) functions
\[
\psi_{\alpha, N}: \{ \pm 1\}^{N} \rightarrow \mathbb{R}, \hspace{3mm} \psi_{\alpha, N}(x) = \sign{(x_{1} + \ldots + x_{N} - \alpha)}, \hspace{3mm} 0\leq \alpha < N.
\]
Indeed, we prove in Theorem \ref{Theo:BooleanRadiusSignFunction} that there is an absolute constant $C> 0$ such that
\[
C^{-1} \frac{1}{\alpha + \sqrt{N}} \leq \rho(\psi_{\alpha,N}) \leq C \frac{1}{\alpha + \sqrt{N}}
\]
for every $0 \leq \alpha <N$. The strategy of the proof actually allows to give the precise asymptotic decay of the Boolean radius of the majority function
$\Maj_{N} = \psi_{N,0}$ for $N$ odd (Corollary~\ref{Coro:BooleanRadiusMaj}).

\section{Case 1: All functions on the Boolean cube}

\label{sec:nVariables}

The following theorem gives the precise Boolean radius for all functions on the $N$-dimensional
Boolean cube $\{\pm{1}\}^N$.

\begin{Theo} \label{Theo:BooleanRadius} For each $N \in \N$
\[
\rho(\mathcal{B}_{N}) = 2^{1/N} - 1.
\]
In particular, $$\lim_{N\to \infty} N \rho(\mathcal{B}_{N}) = \log 2\,.$$
\end{Theo}

\vspace{2 mm}

Let us compare this result with what we already explained in the introduction for polynomials
$f(z) = \sum_{\alpha \in \mathbb{N}_0^N} \widehat{f}(\alpha) z^\alpha$ on the $N$-dimensional torus $\mathbb{T}^N$.
In the Boolean case we have the precise estimation of  $\rho(\mathcal{B}_{N})$ for all $N$, whether in the Bohr case
this is only known for $N=1$, namely $K_1 = 1/3$ (see again \eqref{power}). One of the key ingredients of Bohr's
original proof of \eqref{power} from \cite{Bohr} (more precisely the modification of Bohr's proof by M.~Riesz, I.~Schur,
and F.~Wiener) is a~result of F.~Wiener which states that for every holomorphic function $f$  with Taylor coefficients
$c_n$ on the open unit disc $\mathbb{D}$ with $\|f\|_\infty \leq 1$ we have
\begin{equation} \label{Fritz} |c_n| \leq  1-|c_0|^2 \leq 2 (1-|c_0|), \quad  n \in \mathbb{N}\,.
\end{equation}
The key point in the proof of Theorem
\ref{Theo:BooleanRadius} is that in the Boolean case we have a~stronger F.~Wiener type result as the following
independently interesting lemma shows.

\begin{Lemm}
\label{trick}
For every function $f\colon \{\pm1\}^N \to [-1, 1]$ and each pair of subsets $A, B \subset [N]$ with $A \neq B$
we have that
\begin{align*}
|\widehat{f}(A)| + |\widehat{f}(B)| \leq 1.
\end{align*}
\end{Lemm}

\begin{proof}
Without loss of generality we can assume that there exists $k \in B \setminus A$. Then, for each $x \in \{\pm 1\}^N$ we have
\[
f(x_1, \ldots, x_k, \ldots, x_N) = \sum_{S\subset [N]\setminus\{k\}} \widehat{f}(S) x^S
+ \sum_{S\subset [N],\,k \in S} \widehat{f}(S) x^S
\]
and
\[
f(x_1, \ldots, -x_k, \ldots, x_N) = \sum_{S\subset [N]\setminus\{k\}} \widehat{f}(S) x^S -
\sum_{S\subset [N],\,k \in S} \widehat{f}(S) x^S
\]
Since $\|f\|_{\infty} \leq 1$, both equalities combined yield to
\[
\Big|\sum_{S\subset [N]\setminus\{k\}} \widehat{f}(S) x^S\Big|
+ \Big|\sum_{S\subset [N],\,k \in S} \widehat{f}(S) x^S\Big| \leq 1
\]
and so the claim follows by
\[
|\widehat{f}(A)| + |\widehat{f}(B)|
\leq \mathbb{E}{\Big|\sum_{S\subset [N]\setminus\{k\}} \widehat{f}(S) x^S\Big|}
+  \mathbb{E}{\Big|\sum_{S\subset [N],\,k \in S} \widehat{f}(S) x^S\Big|} \leq 1. \qedhere
\]
\end{proof}

\noindent We are ready to give the proof of the main result.

\begin{proof}[Proof of Theorem \ref{Theo:BooleanRadius}] Let $f\colon \{\pm 1\}^{N} \rightarrow \R$ with $\|f\|_{\infty} \leq 1$.
Applying the estimate  from Lemma \ref{trick}, for every $\rho >0$ we have
\begin{align*}
\sum_{S \subset [N]}{|\widehat{f}(S)| \rho^{|S|}} & = |\widehat{f}(\emptyset)| + \sum_{S \neq \emptyset}{|\widehat{f}(S)| \rho^{|S|}} \\
& = |\widehat{f}(\emptyset)|  + (1 - |\widehat{f}(\emptyset)|) \sum_{k=1}^N \binom{N}{k} \rho^k \\
& = |\widehat{f}(\emptyset)|  + (1 - |\widehat{f}(\emptyset)|) \big((1 + \rho)^N - 1\big).
\end{align*}
If we put $\rho = 2^{1/N} - 1$, then the last term is equal to $1$ and so we conclude that
$
\rho(\mathcal{B}_{N}) \geq 2^{1/N} - 1.
$
Let us estimate now $\rho_{N}$ from above. We denote by $\mathbf{1} \in \{ \pm 1\}^N$ the
element whose entries  are all equal to $1$. Now observe that the function
$F\colon \{\pm 1\}^{N} \rightarrow \{\pm 1\}$ given by
\[
F(x) = \begin{cases} -1,\, & x = \mathbf{1}, \\
1, \, & x \neq \mathbf{1}
\end{cases}
\]
the Fourier expansion of $F$ is given by
\[
F(x) = 1 - \frac{1}{2^{N-1}} + \sum_{S \neq \emptyset}{\frac{-1}{2^{N-1}} \, x^{S}}, \quad\,  x \in \{ \pm 1\}^{N}.
\]
To see this, note that the function $f\colon \{\pm1\}^{N} \rightarrow \mathbb{R}$ given by $f(x)=1$
for $x=\mathbf{1}$ and $f(x)=0$ for $x \neq \mathbf{1}$ has the following Fourier expansion
\[
f(x) = \prod_{n=1}^{N}{\Big( \frac{1 + x_{n}}{2} \Big)} = \frac{1}{2^{N}} \sum_{S\subset [N]} x^S, \quad\, x\in \{\pm1\}^{N}\,.
\]
Since $F(x) = 1 - 2f(x)$ for all $x\in \{\pm1\}^N$, the desired Fourier expansion for $F$ follows.

Finally, it follows from the definition of the Boolean radius $\rho_{N} = \rho(\mathcal{B}_{N})$ that
\[
1 = \| F\|_{\infty} \geq 1 - \frac{1}{2^{N-1}} + \sum_{S \neq \emptyset}{\frac{1}{2^{N-1}} \rho_{N}^{|S|}} =
1 - \frac{1}{2^{N-1}} + \frac{1}{2^{N-1}} \big( (1 + \rho_{N})^N - 1\big).
\]
Thus, $(1 + \rho_N)^N - 1 \leq 1$, and this proves that $\rho_N = 2^{1/N} -1$.
\end{proof}

\section{Case 2: All homogeneous functions}
\label{sec:homogeneous}
The problem of estimating  the Boolean radius of the classes  $\mathcal{B}^{N}_{=m}$ and $\mathcal{B}_{\text{hom}}^{N}$
of homogeneous functions on the Boolean cube $\{\pm{1}\}^N$ is  more delicate than in the case of $\mathcal{B}^N$.

\begin{Theo}\label{Theo:BooleanRadiusmHomogeneous}
There is a constant $C>1$  such that for each $1 \leq m \leq N$
\[
c_m \, N^{\frac{1}{2m}} \binom{N}{m}^{\frac{-1}{2m}}
\leq  \rho(\mathcal{B}_{=m}^{N}) \,\leq\, C_m N^{\frac{1}{2m}} \binom{N}{m}^{\frac{-1}{2m}}\,,
\]
where
\[
c_m = \frac{1}{m^{\frac{1}{2m}}C^{\sqrt{\frac{\log{m}}{m}}}} \quad\, \mbox{ and } \quad\,  C_m = C^{\frac{1}{m}}.
\]
\end{Theo}

Before we turn to consequences of this result and its proof, let us again compare with the analog results known for the $N$-dimensional torus $\mathbb{T}^N$ (instead of the Boolean cube $\{\pm 1\}^N$). If we replace in \eqref{BoKa} all polynomials by all $m$-homogeneous
polynomials and define the $m$th Bohr radius $K_N^{=m}$ as the best $r$ in this inequality, then it is implicitly proved in \cite{Bohrradius} and \cite{Annals} that there is an absolute constant $D > 1$ which  for all $m,N$ satisfies
\begin{equation} \label{Maria}
\frac{1}{D}\Big(\frac{m}{N}\Big)^{\frac{m-1}{2m}} \leq  K_N^{=m} \leq D \Big(\frac{m}{N}\Big)^{\frac{m-1}{2m}}\,.
\end{equation}
and hence  for fixed $m$ the asymptotic decay of the   Bohr radius $K_N^{=m}$ and the Boolean radius $\rho(\mathcal{B}_{=m}^{N})$ coincide.

The control of the decay of $c_m$ and $C_m$ as $m\to \infty$ will give the following.

\begin{Coro}\label{Theo:BooleanRadiusHomogeneous}
\[
\lim_{N \rightarrow \infty } \sqrt{\tfrac{N}{\log{N}}}\,\rho(\mathcal{B}_{{\rm{hom}}}^{N}) =1\,.
\]
\end{Coro}

\vspace{2 mm}

Based on the above formula we come to the essential conclusion that the study of Boolean
radii and Bohr radii show a~remarkable difference. Indeed, note that
\begin{equation} \label{Maria1}
\lim_{N \rightarrow \infty } \sqrt{\tfrac{N}{\log{N}}}\, K_N^{\text{hom}} =1\,,
\end{equation}
where the definition of $K_N^{\text{hom}}$ just means to consider in \eqref{BoKa} only homogeneous polynomials instead of all polynomials on $\mathbb{T}^N$
(see again \cite{Bohrradius} and \cite{Annals}).
So in case of polynomials on $\mathbb{T}^N$ Bohr radii do not distinguish between all polynomials (see  \eqref{best}) and the homogeneous
ones (see \eqref{Maria1}). But in case of functions on $\{\pm 1\}^N$, the preceding corollary and the result from Theorem \ref{Theo:BooleanRadius}
show a~dramatic difference. The deeper reason for this difference  is that the distortion between a function $f: \{ \pm1\}^{N} \rightarrow \mathbb{R}$
and its homogeneous parts $$f_{m}(x)= \sum_{S \subset [N],|S|=m} \widehat{f}(S) x^S, \quad\, 0\leq m \leq N $$  becomes visible if we compare
their supremum norms. We have that $$\|f_m\|_{\infty} \leq C^m \|f\|_{\infty}\,,$$ where the best constant $C$ is at least $\ge \sqrt{2}$ (see \cite{DeMaPe}), whereas by Cauchy inequalities the analog result for complex polynomials on $\mathbb{T}^N$ allows a~constant $1$.
\\

For the proofs of Theorem \ref{Theo:BooleanRadiusmHomogeneous} and Corollary \ref{Theo:BooleanRadiusHomogeneous} we  again need  some preliminary
results. The following inequality is a~recent result from \cite[Theorem 1.1]{DeMaPe}, and we here  state it for  the sake of completeness:
There is an absolute constant $C > 0$ such that for every function $f\colon \{ \pm 1\}^{N} \rightarrow \mathbb{R}$ of degree
$d$ we have that
\begin{equation}\label{Theo:PolyBH}
 \big(\sum_{|S| \leq d}{|\widehat{f}(S)|^{\frac{2d}{d+1}}}\big)^{\frac{d+1}{2d}} \leq C^{\sqrt{d \, \log{d}}} \| f\|_{\infty}.
\end{equation}
The relevant feature for us is that the constant involved has subexponential growth in the sense that
\[
\lim_{d\to \infty}{\big(C^{\sqrt{d \log{d}}}\big)^{\frac{1}{d}}} = 1,
\]
which will play a~fundamental role in the proof of Theorem \ref{Theo:BooleanRadiusHomogeneous}. The second result is of probabilistic nature. In fact, it  is a~consequence of a~celebrated result, and  it will be crucial to  estimate $\rho(\mathcal{B}_{m}^{N})$ from above.

\begin{Lemm}\label{Lemm:SalemZygmund}
For each $N \in \N$ and every family of real numbers $(c_{S})_{S \subset [N]}$, there exists a~choice of signs $(\xi_{S})_{S \subset [N]}$ in $\{\pm 1\}$ such that
\begin{equation}
\Big\| \sum_{S \subset [N]}{\xi_{S} \, c_{S} \, x_{S}} \Big\|_{\infty} \leq 6 \sqrt{\log{2}} \, \sqrt{N} \, \Big( \sum_{S \subset [N]}{|c_{S}|^{2}} \Big)^{\frac{1}{2}}.
\end{equation}
\end{Lemm}

\begin{proof}
We make use of  \cite[p. 68, Theorem 1]{Kahane} (see also \cite[Theorem 8, p.~302]{Blei01}). Every
$f \colon\{ \pm 1\}^{N} \rightarrow \mathbb{R}$ clearly satisfies that
\begin{equation}
\label{equa:halfSupremum}
\sigma\Big\{x \in \{\pm 1\}^{N} \colon |f(x)| \geq \frac{1}{2}\|f\|_{\infty} \Big\} \geq \frac{1}{2^N}.
\end{equation}
If $(\xi_{S})_{S \subset [N]}$ are independent random variables on a probability space $(\Omega, \Sigma, \mathbb{P})$
taking the values $1$ and $-1$ with equal probability, then using \eqref{equa:halfSupremum} in the aforementioned theorem we get that
\[
\mathbb{P} \Big\{ \big\|\sum_{S \subset [N]}{ \xi^{S} \, c_{S} \, \chi_{S} }\big\|_{\infty}
\geq 3 \big( \log{(2^{N+3})} \sum_{S \subset [N]}{|c_{S}|^{2}} \big)^{1/2} \Big\} \leq \frac{1}{2}
\]
which leads to the desired conclusion.
\end{proof}

\begin{proof}[Proof of Theorem \ref{Theo:BooleanRadiusmHomogeneous}]
We start by giving an upper bound for $\rho=\rho(\mathcal{B}_{m}^{N})$, for which we make use of Lemma \ref{Lemm:SalemZygmund}: considering the family $(c_{S})_{S \subset [N]}$ where $c_{S} = 1$ if $|S| = m$ and $c_{S}=0$ otherwise, there exist $(\xi_{S})_{S \subset [N]}$ in $\{ \pm 1 \}$ such that
\[
\rho^{m} \binom{N}{m} =\sum_{|S| = m}{\rho^{m}} \leq \Big\| \sum_{|S| = m}{\xi_{S} \, x_{S}} \Big\|_{\infty} \leq 6 \sqrt{\log{2}} \, \sqrt{N} \, \binom{N}{m}^{\frac{1}{2}}.
\]
Thus the upper estimate follows from this inequality with an absolute constant $C$. We now prove the lover estimate. Combining
Theorem \ref{Theo:PolyBH} with H\"{o}lder's inequality we get that for every $m$-homogeneous function $f\colon \{ \pm 1\}^{N} \rightarrow \mathbb{R}$,
\[
\sum_{|S|=m}{|\widehat{f}(S)|} \leq  \bigg(\sum_{|S|=m}{|\widehat{f}(S)|^{\frac{2m}{m+1}}}\bigg)^{\frac{m+1}{2m}} \binom{N}{m}^{\frac{m-1}{2m}}
\leq  C^{\sqrt{m \log{m}}} \binom{N}{m}^{\frac{m-1}{2m}} \| f \|_{\infty}.
\]
As a consequence
\[
\sum_{|S|=m}{|\widehat{f}(S)| \rho^{m}} \leq \| f\|_{\infty} \hspace{3mm} \mbox{ for } \hspace{3mm} \rho^{m} =  C^{-\sqrt{m \log{m}}} \binom{N}{m}^{\frac{1-m}{2m}}.
\]
Since $\binom{N}{m} \geq (\frac{N}{m})^{m}$ for each integer $1\leq m\leq N$, we conclude that
\[
\rho(\mathcal{B}^{N}_{=m}) \geq C^{-\sqrt{\frac{\log{m}}{m}}} \binom{N}{m}^{\frac{1}{2m^{2}}} \binom{N}{m}^{-\frac{1}{2m}}
\geq \frac{1}{ m^{\frac{1}{2m}} C^{\sqrt{\frac{\log{m}}{m}}} } N^{\frac{1}{2m}} \binom{N}{m}^{-\frac{1}{2m}}
\]
and this completes the proof.
\end{proof}

\begin{proof}[Proof of Corollary \ref{Theo:BooleanRadiusHomogeneous}]
Note that $\mathcal{B}_{\text{hom}}^{N}$ is the union of all families $\mathcal{B}^{N}_{=m}$, which immediately leads to the relation
\[
\rho(\mathcal{B}_{\text{hom}}^{N}) = \inf_{0 \leq m \leq N}{\rho(\mathcal{B}^{N}_{=m})}.
\]
Using the estimations $\sqrt{2 \pi} n^{n+\frac{1}{2}} e^{-n} \leq n! \leq e n^{n + \frac{1}{2}} e^{-n}$ valid for all $n \geq 1$,
we can bound for each $1 \leq m \leq N$
\[
\binom{N}{m} \geq \frac{\sqrt{2 \pi}}{e^{2}} \left( \frac{N}{m} \right)^{m} \, \left( \frac{N}{N-m} \right)^{N-m} \, \sqrt{\frac{N}{m (N-m)}}.
\]
Applying this inequality in the upper estimation of Theorem \ref{Theo:BooleanRadiusmHomogeneous} we get that
\[
\rho(\mathcal{B}^{N}_{=m}) \leq C_{m} N^{\frac{1}{2m}} \left( \frac{e^{2}}{\sqrt{2 \pi}} \right)^{\frac{1}{2m}} \sqrt{\frac{m}{N}} \left( 1 - \frac{m}{N}\right)^{\frac{1}{2} (\frac{N}{m} - 1)} \left( \frac{m (N-m)}{N} \right)^{\frac{1}{4m}}
\]
Taking $m =[\log{N}]$, the previous inequality leads to
\begin{equation*}
\rho(\mathcal{B}_{\text{hom}}^{N}) \leq \rho(\mathcal{B}^{N}_{=[\log{N}]}) \leq \sqrt{\frac{\log{N}}{N}} (1 + o(1)).
\end{equation*}
To prove the converse estimation, we recall now the lower estimation given in Theorem \ref{Theo:BooleanRadiusmHomogeneous} which combined with the inequality $\binom{N}{m} \leq \big( \frac{N e}{m}\big)^{m}$ valid for each $1 \leq m \leq N$ yields that
\begin{equation}\label{equa:auxHomogeneousBoolean1} \frac{\rho(\mathcal{B}^{N}_{=m})}{ \sqrt{\frac{\log{N}}{N}}}  \geq \frac{N^{\frac{1}{2m}} \sqrt{m}}{C_{m} \sqrt{\log{N}} \sqrt{e}} = \frac{1}{C_{m}} \exp{\left( \frac{\log{N}}{2m} - \frac{1}{2} + \frac{\log{m}}{2} -\frac{\log{\log{N}}}{2} \right)}.
\end{equation}
Let $\varepsilon > 0$, and fix $m_{0} \in \N$ satisfying that $C_{m}^{-1} > 1 - \varepsilon$ whenever $m \geq m_{0}$. Using \eqref{equa:auxHomogeneousBoolean1} we deduce the existence of some $N_{0} \in \N$ such that
\begin{equation}\label{equa:auxBooleanRadiusHomogeneous2}
\inf_{ 0 \leq m \leq m_{0}}{\rho(\mathcal{B}^{N}_{=m})} \geq \sqrt{\frac{\log{N}}{N}} (1 - \varepsilon), \quad\, N \geq N_{0}.
\end{equation}
On the other hand, the function
\[
\mathbb{R}_{+}  \ni x \longmapsto \frac{\log{N}}{2x} - \frac{1}{2} + \frac{\log{x}}{2} - \frac{\log{\log{N}}}{2}
\]
reaches its minimum at $x=\log{N}$ with value equal to zero. Thus, again by \eqref{equa:auxHomogeneousBoolean1}
\begin{equation}\label{equa:auxBooleanRadiusHomogeneous3}
\inf_{m_{0} \leq m \leq N} \rho(\mathcal{B}^{N}_{=m}) \geq \sqrt{\frac{\log{N}}{N}} \inf_{m_{0} \leq m \leq N}{\frac{1}{C_{m}}} \geq \sqrt{\frac{\log{N}}{N}} (1- \varepsilon)
\end{equation}
for each $N$. Combining \eqref{equa:auxBooleanRadiusHomogeneous2} and \eqref{equa:auxBooleanRadiusHomogeneous3} we conclude that
\[  \rho(\mathcal{B}^{N}_{\text{hom}}) \geq \inf_{0 \leq m\leq N}{ \rho(\mathcal{B}_{=m}^{N})} \geq \sqrt{\frac{\log{N}}{N}} (1- \varepsilon), \quad\,
N \geq N_{0}.
\]
\end{proof}

\section{Case 3: All degree-$d$ functions}
\label{sec:degreed}
Cauchy inequalities allow to extend \eqref{Maria} to the degree-$d$ case
\begin{equation} \label{Maria2}
\frac{1}{D}\Big(\frac{m}{N}\Big)^{\frac{d-1}{2d}} \leq  K_N^{\leq d} \leq D \Big(\frac{m}{N}\Big)^{\frac{d-1}{2d}}\,,
\end{equation}
where the definition of $K_N^{\leq d}$ by now is obvious. Unfortunately, the information we have on the Boolean case is less precise.

\begin{Theo}\label{Theo:BooleanRadiusDegree}
Let $1 \leq d \leq N$. Then, there are absolute constants $c_{d}$, $C_{d} > 0$ satisfying that
\begin{equation}\label{equa:BooleanRadiusDegree}
c_{d}\frac{1}{N^{\frac{1}{2}}} \leq \rho(\mathcal{B}^{N}_{\leq d}) \leq C_{d} \frac{1}{N^{\frac{d-1}{2d}}}.
\end{equation}
\end{Theo}

\vspace{2 mm}

\noindent
 We point out that the upper bound will follow easily from the results in the previous
section, the lower estimate requires some preliminary work that may be of independent interest.

\begin{Theo} \label{Theo:MAINsumlessthanone}
For every map $f:\{ \pm 1\}^{N} \rightarrow \mathbb{R}$ and each $x \in \{ \pm 1\}^{N}$ we have that
\begin{equation}\label{equa:MAINsumlessthanone}
\Big| \widehat{f}(\emptyset) + \frac{1}{2} \sum_{S \neq \emptyset}{\widehat{f}(S) \, x^{S}} \Big| + \Big| \frac{1}{2}
\sum_{S \neq \emptyset}{\widehat{f}(S) \, x^{S}} \Big| \leq \| f\|_{\infty}.
\end{equation}
\end{Theo}

\begin{proof}
Fix $x \in \{ \pm 1\}^{N}$. Given $A \subset [N]$, if we denote by $\tilde{x}$ the element defined as $\tilde{x}_{n} = x_{n}$
if $n \notin A$ and $\tilde{x}_{n} = -x_{n}$ if $n \in A$; then we have that
\begin{align*}
& | f(x) | = \Big| \sum_{|S\cap A| \mbox{ even }}{\widehat{f}(S) \, x^{S}} + \sum_{|S\cap A| \mbox{ odd }}{\widehat{f}(S) \, x^{S}} \Big|,\\
& |f(\tilde{x})| = \Big| \sum_{|S\cap A| \mbox{ even }}{\widehat{f}(S) \, x^{S}} - \sum_{|S\cap A| \mbox{ odd }}{\widehat{f}(S) \, x^{S}} \Big|.
\end{align*}
Both things together lead to
\[ \Big| \sum_{|S\cap A| \mbox{ even }}{\widehat{f}(S) \, x^{S}} \Big| + \Big| \sum_{|S\cap A| \mbox{ odd }}{\widehat{f}(S) \, x^{S}} \Big| \leq \| f\|_{\infty}. \]
If we sum over all $A \subset [N]$ and divide by $2^{N}$, then we will get by the triangular inequality, that
\begin{equation} \label{equa:sumlessthatone}
\Big|\frac{1}{2^N} \sum_{A \subset[N]}{\sum_{|S\cap A| \mbox{ even }}{\widehat{f}(S) \, x^{S}}} \Big| + \Big| \frac{1}{2^N} \sum_{A \subset[N]}{\sum_{|S\cap A| \mbox{ odd }}{\widehat{f}(S) \, x^{S}}} \Big| \leq \| f\|_{\infty}.
\end{equation}
We can rewrite the double sums in a more affordable sum
\[ \sum_{A \subset[N]}{\sum_{|S\cap A| \mbox{ odd }}{\widehat{f}(S) \, x^{S}}} = \sum_{S \neq \emptyset}{\widehat{f}(S) \, x^{S} |\{ A \subset [N]\colon |A \cap S| \, \mbox{ is odd} \}|}. \]
To count the number of subsets $A \subset [N]$ such that $|A \cap S|$ is odd, note that such set is of the form $A = A_{1} \cup A_{2}$ where $A_{1} \subset [N] \setminus S$ and $A_{2} \subset S$ satisfies that $|A_{2}|$ is odd. Since the number of subsets $A_{2} \subset S \neq \emptyset$ with an odd number of elements is precisely $2^{|S|-1}$, we deduce that
\[
|\{ A \subset [N]\colon |A \cap S| \, \mbox{ is odd} \}| = 2^{N - |S|} \cdot 2^{|S|-1} = 2^{N-1}.
\]
Therefore
\begin{equation}
\label{equa:oddPart}
\frac{1}{2^N}\sum_{A \subset[N]}{\sum_{|S\cap A| \mbox{ odd }}{\widehat{f}(S) \, x^{S}}} = \frac{1}{2} \sum_{S \neq \emptyset}{\widehat{f}(S) \, x^{S}}.
\end{equation}
Following the same strategy, we have that
\begin{equation}
\label{equa:evenPart}
\frac{1}{2^N} \sum_{A \subset[N]}{\sum_{|S\cap A| \mbox{ even }}{\widehat{f}(S) \: x^{S}}} = \widehat{f}(\emptyset) + \frac{1}{2} \sum_{S \neq \emptyset}{\widehat{f}(S) \, x^{S}}\,.
\end{equation}
Replacing \eqref{equa:oddPart} and \eqref{equa:evenPart} in \eqref{equa:sumlessthatone}, we get the result.
\end{proof}

If we take expectations in \eqref{equa:MAINsumlessthanone} and bound the first integral by $|\widehat{f}(\emptyset)|$, then we immediately obtain
the following result.

\begin{Coro}
\label{Coro:CaratheodoryTypeIneq}
For every function $f:\{ \pm 1\}^{N} \rightarrow [-1,1]$ we have that
\begin{equation}\label{equa:CaratheodoryTypeIneq}
\mathbb{E}{\Big| \sum_{S \neq \emptyset}{\widehat{f}(S) \: x^{S}} \Big|} \leq 2 (1 - |\widehat{f}(\emptyset)|).
\end{equation}
\end{Coro}

The previous corollary cannot be improved in the sense that we cannot find $p > 1$ and $\gamma > 0$ such that
\begin{equation} \label{equa:pIntegralCaratheodory}
\Big(\mathbb{E}{\Big| \sum_{S \neq \emptyset}{\widehat{f}(S) \, x_{S}} \Big|^{p} }\Big)^{\frac{1}{p}} \leq \gamma (1 - |\widehat{f}(\emptyset)|)
\end{equation}
for every function $f\colon \{\pm 1\}^{N} \rightarrow \R$. Indeed, take $f:\{\pm 1\}^{N} \rightarrow \{ \pm 1\}$ with $\sigma(f=1) = \lambda$ and $\sigma(f=-1) = 1 - \lambda$ for some $0 < \lambda < 1/2$. Then, $\widehat{f}(\emptyset) = 2 \lambda -1$ and
\[
\theta(x):=f(x) - \widehat{f}(\emptyset) = \begin{cases} 2 - 2\lambda, & f(x)=1,\\ - 2\lambda, & f(x)=-1. \end{cases}
\]
Therefore,
\begin{align*}
\mathbb{E}{|\theta(x)|^{p}} & = (2 -2 \lambda)^{p} \lambda + (2\lambda)^{p} (1 - \lambda).
\end{align*}
and replacing this value in \eqref{equa:pIntegralCaratheodory} we get that
\begin{equation}\label{equa:pIntegralCaratheodoryAux}
(2 -2 \lambda)^{p} \lambda + (2\lambda)^{p} (1 - \lambda) \leq \gamma^{p} (2 \lambda)^{p},
\end{equation}
which leads to
\[
(1 - \lambda)^{p} \leq (1 - \lambda)^{p} + \lambda^{p-1} (1 - \lambda) \leq \gamma^{p} \lambda^{p-1}.
\]
But for $N$  large enough we could take $\lambda$ small enough to contradict
the previous inequality for the given $\gamma$.

\begin{Rema}
It also follows from \eqref{equa:pIntegralCaratheodoryAux} for the case $p=1$ that the constant $2$ in
\eqref{equa:CaratheodoryTypeIneq} is optimal.
\end{Rema}

\begin{Coro}
\label{Coro:CaratheodoryTypeIneq2}
Let $f\colon \{ \pm 1\}^{N} \rightarrow [-1,1]$ be a function of degree $d$. Then
\begin{equation}\label{equa:CaratheodoryTypeIneq2}
\Big(\sum_{0 < |S| \leq d}{|\widehat{f}(S)|^{2}}\Big)^{\frac{1}{2}} \leq 2 e^{d} (1 - |\widehat{f}(\emptyset)|).
\end{equation}
\end{Coro}

\begin{proof}
Using \cite[Theorem 9.22]{BooleanRyan}
\label{Coro:comparingNorms} we have that $f:\{ \pm 1\}^{N} \rightarrow \mathbb{R}$ is a function of degree $d$, then $\| f\|_{2} \leq e^{d} \| f\|_{1}$
(see again \eqref{p-norm} for the definition of the norms). Combining this fact with Corollary \ref{Coro:CaratheodoryTypeIneq} we get the desired conclusion.
\end{proof}

\begin{proof}[Proof of Theorem \ref{Theo:BooleanRadiusDegree}]
The upper bound is consequence of the trivial inequality $\rho(\mathcal{B}^{N}_{\leq d}) \leq \rho(\mathcal{B}^{N}_{d})$ and Theorem \ref{Theo:BooleanRadiusmHomogeneous}. To prove the lower bound, let us fix a (normalized) function $f:\{ \pm 1\}^{N} \rightarrow [-1,1]$ of degree $d$ with $\| f\|_{\infty}= 1$. Put $\rho = C/\sqrt{N}$ where $0<C<1$ is a constant. Using H\"{o}lder's inequality together with Corollary \ref{Coro:CaratheodoryTypeIneq2} we obtain that
\begin{align*}
\sum_{S \subset [N]}{|\widehat{f}(S)| \rho^{|S|}} & \leq |\widehat{f}(\emptyset)| + \Big( \sum_{0<|S| \leq d}{|\widehat{f}(S)|^{2}} \Big)^{\frac{1}{2}} \bigg(\sum_{m=1}^{d}{\rho^{2m} \binom{N}{m}} \bigg)^{\frac{1}{2}}\\
& \leq |\widehat{f}(\emptyset)| + 2e^{d}(1 - |\widehat{f}(\emptyset)|) \Big( \sum_{m=1}^{d}{(N \rho^{2})^{m}} \Big)^{\frac{1}{2}}\\
& \leq |\widehat{f}(\emptyset)| + 2 e^{d} (1 - |\widehat{f}(\emptyset)|) \frac{C}{\sqrt{1 - C^{2}}}\,.
\end{align*}
It then follows that we can find a constant $C=c_{d}$ such that the previous sum is less than or equal to one. This yields that
\[
\rho(\mathcal{B}^{N}_{\leq d}) \geq \frac{c_{d}}{\sqrt{N}}.\qedhere
\]
\end{proof}

The gap between the exponents of $N$ in the upper and lower estimations of \eqref{equa:BooleanRadiusDegree} is notorious. We strongly
believe that $1/N^{\frac{d-1}{2d}}$ is the right order of decay in view of the result for the $d$-homogeneous case. In this sense,
+we pose the following related question.

\begin{Ques}\label{Ques:WienerTypeFull}
Given $d \in\N$, is there a constant $C_{d} > 0$ such that every function $f:\{ \pm 1\}^{N} \rightarrow [-1,1]$ of degree $d$
\begin{equation}\label{equa:CaratheodoryTypeDegree}
\Big\| \sum_{S \neq \emptyset}{\widehat{f}(S) x^{S}} \Big\|_{\infty} \leq C_{d} (1 - |\widehat{f}(\emptyset)|)\,?
\end{equation}
\end{Ques}

A positive answer to the previous question would imply that we can replace $N^{\frac{1}{2}}$ by $N^{\frac{d-1}{2d}}$ in the lower bound of \eqref{equa:BooleanRadiusDegree}. Indeed, we could refine the second part of the proof of Theorem \ref{Theo:BooleanRadiusDegree} using the inequality
\[ \sum_{0<|S|\leq d}{|\widehat{f}(S)|\rho^{|S|}} \leq \Big(\sum_{0 < |S| \leq d}{|\widehat{f}(S)|^{\frac{2d}{d+1}}}\Big)^{\frac{d+1}{2d}}
\bigg(\sum_{m=1}^{d}{\rho^{m \frac{2d}{d-1}}} \binom{N}{m}\bigg)^{\frac{d-1}{2d}}\,,
\]
and then applying \eqref{Theo:PolyBH} together with \eqref{equa:CaratheodoryTypeDegree}.

Let us note that Question \ref{Ques:WienerTypeFull} is equivalent to the following:

\begin{Ques}\label{Ques:WienerTypeHomogeneousl}
Given $d \in\N$, is there a constant $C_{d} > 0$ such that every function $f:\{ \pm 1\}^{N} \rightarrow [-1,1]$ of degree $d$ and each $1 \leq m \leq d$
we have that
\begin{equation*}\label{equa:CaratheodoryTypeDegree}
\Big\| \sum_{|S| = m}{\widehat{f}(S) x^{S}} \Big\|_{\infty} \leq C_{d} (1 - |\widehat{f}(\emptyset)|)\,?
\end{equation*}
\end{Ques}
The previous condition is very similar to the statement of Wiener's theorem for complex polynomials in several variables $Q(z) = \sum_{m=0}^{d}{Q_{m}(z)}$, where $Q_{m}$ denotes the $m$-homogeneous part, namely that for each $0 \leq m \leq d$  $$\| Q_{m} \|_{\infty} \leq 1 - |Q_{0}|^{2} \leq 2 (1 - |Q_{0}|);$$
see \eqref{Fritz} for the one dimensional case. However, in the Boolean case we cannot expect that the constants $C_{d}$ in Question \ref{Ques:WienerTypeHomogeneousl} satisfy $C_{d} \leq C$ for an absolute constant $C$. Indeed,  otherwise
\begin{align*}
\sum_{0 < |S| \leq N}{|\widehat{f}(S)| \rho^{|S|}} & \leq \sum_{m=1}^{N}{\big( \sum_{|S|=m}{|\widehat{f}(S)|^{\frac{2m}{m+1}}} \big)^{\frac{m+1}{2m}} \rho^{m} \binom{N}{m}^{\frac{m-1}{2m}}}\\
&  \leq C (1 - |\widehat{f}(\emptyset)|) \sum_{m=1}^{N}{\Big( \frac{ \rho \, \sqrt{N} \sqrt{e}}{\sqrt{m}} \Big)^{m}}\,,
\end{align*}
and then $\rho(\mathcal{B}^{N}) \geq \Omega(1/\sqrt{N})$. But by  Theorem \ref{Theo:BooleanRadius} we have $\rho(\mathcal{B}^{N}) \leq O(1/N)$, a~contradiction.

\section{Case 4: The class $\mathcal{B}^{N}_{\delta}$}
\label{sec:biased}

\noindent The main result of this final section reads as follows.

\begin{Theo}\label{Theo:BooleanRadiusUnbiased}
There is an absolute constant $C > 0$ such that for each $N \in \N$ and every $\frac{1}{2^{N}} \leq \delta \leq 1$ we have that
\begin{equation}\label{equa:BooleanRadiusUnbiased}
\frac{C^{-1}}{\sqrt{N} \sqrt{\log{(2/\delta)}}} \leq \rho(\mathcal{B}^{N}_{\delta}) \leq \frac{C}{\sqrt{N} \sqrt{\log{(2/\delta)}}}.
\end{equation}
\end{Theo}

The proof of the theorem is separated into two parts. First we deal with the left-hand side inequality of \eqref{equa:BooleanRadiusUnbiased}, which is based on the use of the hypercontractivity  inequalities of Bonami-Gross.
Recall that for  $-1 < \rho < 1$  the noise operator $T_\rho$  assigns to every function $f:\{\pm 1\}^N \rightarrow \mathbb{R}$ on the Boolean cube the map
\[ T_{\rho}f:\{ \pm 1\}^{N} \longrightarrow \mathbb{R}, \quad\, T_{\rho}f(x) := \sum_{S \subset [N]}{\widehat{f}(S) \rho^{|S|} x^{S}}. \]
An important feature of this operator is presented in the next celebrated result.
\begin{Theo}[Bonami-Gross]
For every $1 \leq p \leq q \leq \infty$
\begin{equation}\label{Theo:hypercontractive}
 \| T_{\rho}f\|_{q} \leq \| f\|_{p} \quad\, \mbox{whenever $\rho \leq \sqrt{\frac{p-1}{q-1}}$}.
\end{equation}
\end{Theo}

On the other hand, to prove  the right-hand side inequality of \eqref{equa:BooleanRadiusUnbiased} we are giving precise estimations
of certain type of threshold functions belonging to $\mathcal{B}^{N}_{\delta}$.\\

\subsection*{Lower bound}

\begin{Theo}\label{Theo:biasedBooleanRadius}
Let $f\colon \{ \pm 1\}^{N} \longrightarrow \mathbb{R}$ and $0<\delta \leq 1$ with
$ |\mathbb{E}[f]| \leq (1- \delta) \|f\|_{\infty} $. Then, we have that
\[
\rho(f) \geq  \frac{1}{5 \, \sqrt{N} \,  \sqrt{\log{(2/\delta)}}}.
\]
\end{Theo}

\begin{proof}
Without loss of generality we can assume that  $f\colon \{\pm 1\}^N \rightarrow \mathbb{R}$ satisfies $\| f\|_{\infty} = 1$ and $\mathbb{E}[f] = 1 - \delta$. We follow an strategy similar to the proof of the small-set expansion theorem and the level-$k$ inequalities (see \cite[p. 259]{BooleanRyan}). We first claim that for every $0<\varepsilon <1$ and $0 \leq m \leq N$ we have that
\begin{equation}\label{equa:biasedBooleanHomogeneousAux}
\Big( \sum_{|S| = m}{|\widehat{f}(S)|^2} \Big)^{\frac{1}{2}}  \leq 2 \, \varepsilon^{-\frac{m}{2}} \, \left( \frac{\delta}{2} \right)^{\frac{1}{1 + \varepsilon}}.
\end{equation}
To see this, let us define $g\colon \{\pm 1\}^N \rightarrow [0,1]$ by $g(x) = (1 - f(x))/2$. Notice that
\begin{align}\label{equa:Fouriercoeff}
 \widehat{g}(S) = - \frac{1}{2} \widehat{f}(S), \quad\, \emptyset \neq S \subset [N].
\end{align}
Moreover, for every $p \geq 1$ we have that
\begin{equation}\label{equa:p-normBound}
\Big( \mathbb{E}[ |g(x)|^p ] \Big)^{\frac{1}{p}} \leq \Big( \mathbb{E}[g(x)] \Big)^{\frac{1}{p}} = \left( \frac{\delta}{2} \right)^{\frac{1}{p}}.
\end{equation}
Combining \eqref{equa:Fouriercoeff}, the Hypercontractivity Theorem \ref{Theo:hypercontractive} with $p=1 + \varepsilon< 2$
and \eqref{equa:p-normBound}, we obtain that
\begin{align*}
\sum_{|S|=m}{|\widehat{f}(S)|^2} & = 4 \sum_{|S|=m}{|\widehat{g}(S)|^2} \leq \frac{4}{\varepsilon^m} \sum_{S \subset [N]}{|\widehat{g}(S)|^2 \varepsilon^{|S|}}\\
& \leq \frac{4}{\varepsilon^m} \left( \mathbb{E}{|g(x)|^p} \right)^{\frac{2}{p}} \leq \frac{4}{\varepsilon^m} \left( \frac{\delta}{2} \right)^{\frac{2}{1 + \varepsilon}}.
\end{align*}
This finishes the proof of the claim. Applying now \eqref{equa:biasedBooleanHomogeneousAux} we get that for any $0< \rho$ and $0<\varepsilon<1$ it holds that
\begin{align*}
\sum_{S \neq \emptyset}{|\widehat{f}(S)| \, \rho^{|S|}} & = \sum_{m=1}^{N}{ \rho^{m} \, \sum_{|S| = m}{|\widehat{f}(S)|}} \leq \sum_{m=1}^{N}{ \rho^{m} \, \sqrt{\binom{N}{m}} \Big(\sum_{|S|=m}{|\widehat{f}(S)|^2}\Big)^{\frac{1}{2}} }\\
& \leq \sum_{m=1}^{N}{\rho^{m} \, \sqrt{\binom{N}{m}} \, 2 \, \varepsilon^{-\frac{m}{2}} \, \Big(\frac{\delta}{2}\Big)^{\frac{1}{1+ \varepsilon}} }\\
& \leq \sum_{m=1}^{N}{ \rho^{m} \, e^{\frac{m}{2}} \, \frac{N^\frac{m}{2}}{m^{\frac{m}{2}}} \, 2 \, \varepsilon^{-\frac{m}{2}} \left( \frac{\delta}{2} \right)^{\frac{1}{1+\varepsilon}} }=  2 \, \left( \frac{\delta}{2} \right)^{\frac{1}{1+\varepsilon}} \, \sum_{m=1}^{N}{ \Big( \frac{\rho \sqrt{e} \sqrt{N}}{\sqrt{\varepsilon} \sqrt{m}} \Big)^{m}}.
\end{align*}
Taking
\[
\varepsilon = \frac{\log{2}}{\log{(2/\delta)}} \hspace{3mm} \, \mbox{and} \, \hspace{3mm} \rho = \frac{1}{5 \, \sqrt{N} \,  \sqrt{\log{(2/\delta)}}},
\]
we deduce that
\[
\left(\frac{\delta}{2}\right)^{\frac{1}{1+ \varepsilon}} = \frac{\delta}{2} \exp{\left( \frac{\varepsilon}{1 + \varepsilon} \log{\frac{2}{\delta}} \right)} = \frac{\delta}{2} \exp{\left( \frac{ \log{(2/\delta)} \, \log{2}}{\log{(2/\delta)} + \log{2}}  \right)} \leq \delta.
\]
Using now that $\widehat{f}(\emptyset) = \mathbb{E}[f] = 1 - \delta$, the previous estimations yield to
\[
\sum_{S \neq \emptyset}{|\widehat{f}(S)| \, \rho^{|S|}} \leq 2 \, \delta \, \sum_{m=1}^{N}{\Big( \frac{\sqrt{e}}{5 \sqrt{m} \sqrt{\log{2}}} \Big)^{m}} \leq \delta = 1 - |\widehat{f}(\emptyset)|.
\qedhere
\]
\end{proof}

\subsection*{Threshold type functions}

\noindent Finally, we are going to show that the estimation in Theorem \ref{Theo:biasedBooleanRadius} is somehow optimal. For that we give precise estimates of the  Boolean radius of the following two-parametric functions:
\[
\psi_{N, \alpha}: \{\pm 1\}^{N} \rightarrow \{\pm 1\}, \hspace{3mm} \psi_{N, \alpha}(x) = \sign{(x_{1} + \ldots + x_{N} - \alpha)}\,,
\]
where $\alpha \in \mathbb{R}$, $N \in \N$, and as usual $\sign(y) = 1$ for  $y \geq 0$ and $-1$ for  $y < 0$. Different values of $\alpha$ may generate the same function. We will restrict to the case in which $\alpha \geq 0$. For some results we will also work assuming that $N-\alpha$ is an odd natural number. This way, we guarantee that different values of $\alpha$ lead to different functions and we also avoid that $\sign(0)$ appears. Let us point out that the functions $\psi_{N, \alpha}$ belong to a wider class of functions known as linear threshold function (see \cite[Chapter 5]{BooleanRyan}). The aim is to show the following result.

\begin{Theo}\label{Theo:BooleanRadiusSignFunction}
For each $N \in \N$ and $0 \leq \alpha < N$ we have that
\begin{equation}
\label{equa:TheoBooleanRadiusSignFunctionMain}
C^{-1} \frac{1}{\alpha + \sqrt{N}} \leq \rho(\psi_{N, \alpha}) \leq C \frac{1}{\alpha + \sqrt{N}}\,,
\end{equation}
where $C$ is an absolute constant independent of $\alpha$ and $N$.
\end{Theo}

\noindent The proof is involved and requires a series of preliminary results. We start with a~technical
estimate for the Boolean radius with some extra assumption on $\alpha$.
For $N \in \N$ with  $0 \leq \alpha < N$ let $G\colon \mathbb{R}_{+} \rightarrow \mathbb{R}_{+}$ be the function  given by
\begin{equation} \label{G}
G(r) := \sup_{z \in \mathbb{T}}{|1 + z r|^{\frac{N+\alpha - 1}{2}} |1- z r|^{\frac{N-\alpha - 1}{2}}}, \quad\, r\in \mathbb{R}_{+}\,,
\end{equation}
and use it to define $\mathcal{I}\colon \mathbb{R}_{+}  \rightarrow \mathbb{R}_{+}$  by
\begin{equation} \label{I}
\mathcal{I}(\rho) = \int_{0}^{\rho} G(r) \: d r, \quad\, \rho\in \mathbb{R}_{+}\,.
\end{equation}

Before we state and prove the  required preliminary result we define, following Guadarrama \cite{Guad},
the Bohr radius for the class $\mathcal{P}_N$ of all complex polynomials of degree at most $N$ by
\begin{equation} \label{Fournier}
R_{N}:= \sup\bigg\{r>0: \, \sum_{k=0}^{N} |c_k(p)|r^k \leq \|p\|_{\mathbb{D}},\,\,\, p(z) = \sum_{k=0}^N c_k(p) z^k, \,\, z\in \mathbb{D}\bigg\}.
\end{equation}
It was proved in \cite{Guad} that there exist absolute constants $c_1$, $c_2>0$ such that for large enough $N$,
\[
\frac{1}{3} + \frac{c_1}{3^{N/2}} < R_N < \frac{1}{3} + c_2 \frac{\log N}{N}.
\]
Dryanov and Fournier \cite{DryFour} proved the following remarkable estimate: For every polynomial $p\in \mathcal{P}_N$,
$p(z)= \sum_{k=0}^{N} c_k z^k$ we have
\[
|c_k| \leq 2 \cos\bigg(\frac{\pi}{[\frac{N}{k}] + 2}\bigg) \big(\|p\|_{\mathbb{D}} - |c_0|\big),\quad\,  1\leq k\leq N,
\]
where as usual $[\frac{N}{k}]$ denotes the integer part of $\frac{N}{k}$. This estimations implies that the unique root
$t_N$ in $(0,1]$ of the equation
\[
\sum_{k=0}^{N} \cos\bigg(\frac{\pi}{[\frac{N}{k}] + 2}\bigg)t^k = \frac{1}{2}
\]
satisfies $0<t_N \leq R_N$. We refer to a~remarkable Fournier's article \cite{Fournier}, where among others it is shown that
in fact $t_N <R_N$ for each $N$.

\begin{Theo}\label{Theo:BestEstimationRadiusSignFunctions}
For $N \in \N$ and $0 \leq \alpha < N$ such that $N-\alpha$ is odd let $\rho = \rho_{\alpha, N}$ be the Boolean radius of $\psi = \psi_{\alpha, N}$. Then
\begin{equation}\label{equa:BestEstimationRadiusSignFunctions1}
\mathcal{I}(\rho) \leq \frac{1}{N  \binom{N-1}{\frac{N - \alpha - 1}{2}}} \sum_{n \leq \frac{N - \alpha -1}{2}}{\binom{N}{n}}
\leq R_N \,\mathcal{I}\Big(\frac{\rho}{R_N}\Big)\,,
\end{equation}
where $R_N$ is the Bohr radius defined in \eqref{Fournier}.
\end{Theo}

\begin{proof}
We start by establishing that for every complex number $z$ it holds that
\begin{equation}\label{equa:LemmFourierSpectrumSignFunctionsAux1}
\sum_{n=1}^{N}{\binom{N-1}{n-1}} \widehat{\psi}([n]) z^{n-1} = \binom{N-1}{\frac{N-\alpha - 1}{2}}
\frac{1}{2^{N-1}} (1 + z)^{\frac{N+\alpha - 1}{2}} (1 - z)^{\frac{N-\alpha - 1}{2}}.
\end{equation}
For that, we make use of the $N$-th discrete derivative of a~function on the Boolean cube \cite[Section 2.2]{BooleanRyan},
which applied to $\psi$ gives a function $D_{N}\psi : \{ -1,1\}^{N-1} \longrightarrow \mathbb{R}$ defined as
\begin{align*}
D_{N}\psi(x_{1}, \ldots, x_{N-1}) & := \frac{\psi(x_{1}, \ldots, x_{N-1}, 1) - \psi(x_{1}, \ldots, x_{N-1}, -1)}{2}\\
& = \begin{cases} 1 & \mbox{ if $ x_{1} + \ldots + x_{N-1} = \alpha$}\\ 0 & \mbox{ otherwise } \end{cases}
\end{align*}
The Fourier expansion of $D_{N}\psi$ can be explicitly calculated in terms of the Fourier coefficients of $\psi$
(see \cite[p. 30, Prop. 2.9]{BooleanRyan} for a detailed proof). Using moreover that $\psi$ is invariant by permutations
of its coordinates, which yields that each $\emptyset \neq S \subset [N]$ satisfies  $\widehat{\psi}(S) = \widehat{\psi}([n])$
whether $|S| = n$, we obtain that
\begin{equation}\label{equa:LemmFourierSpectrumSignFunctionsAux3}
D_{N}\psi(x) = \sum_{N \in S \subset [N]}{\widehat{\psi}(S) \, x^{S \setminus \{N\}}} = \sum_{n=1}^{N}{ \widehat{\psi}([n])
\sum_{\substack{N \in S \subset [N]\\ |S| = n }} x^{S \setminus \{ N\}}}.
\end{equation}
Hence, applying the noise operator $T_{r}$ ($-1 < r < 1$) to $D_{N}\psi$ we get that
\begin{equation}\label{equa:LemmFourierSpectrumSignFunctionsAux4}
T_{r}D_{N}\psi(1,\ldots, 1) = \sum_{n=1}^{N}{\binom{N-1}{n-1} \widehat{\psi}([n]) \, r^{n-1}}.
\end{equation}
On the other hand, $T_{r}D_{N}\psi$ has the following probabilistic definition (see \cite[Section 2.4]{BooleanRyan}): if $\mathbf{x}_{i}$ are independent random variables taking values $1$ and $-1$ with probability $\frac{1}{2} + \frac{1}{2} r$ and $\frac{1}{2} - \frac{1}{2} r$ respectively, then
\begin{equation}\label{equa:LemmFourierSpectrumSignFunctionsAux5}
\begin{split}
T_{r}D_{N}\psi(1, \ldots, 1) & = \mathbb{E}{\big[D_{N}\psi(\mathbf{x}_{1}, \ldots, \mathbf{x}_{N-1}) \big]} = \mathbb{P}(\mathbf{x}_{1} + \ldots + \mathbf{x}_{N-1} = \alpha)\\
& = \binom{N-1}{\frac{N- \alpha-1}{2}}\frac{1}{2^{N-1}} \big( 1+ r \big)^{\frac{N+\alpha-1}{2}} \big( 1-r\big)^{\frac{N - \alpha-1}{2}}
\end{split}
\end{equation}
Comparing \eqref{equa:LemmFourierSpectrumSignFunctionsAux4} and \eqref{equa:LemmFourierSpectrumSignFunctionsAux5} we conclude that \eqref{equa:LemmFourierSpectrumSignFunctionsAux1} holds  as both complex polynomials coincide on $r \in (-1,1)$.
It follows from  \eqref{equa:LemmFourierSpectrumSignFunctionsAux1} that
\begin{equation}\label{equa:LemmFourierSpectrumSignFunctionsAux2}
\sup_{z \in \mathbb{T}}{\left| \sum_{n=1}^{N}{\binom{N-1}{n-1} \widehat{\psi}([n]) (r z)^{n-1}} \right|}  = \binom{N-1}{\frac{N-\alpha-1}{2}} \frac{G(r)}{2^{N-1}}, \quad\, r \geq 0\,,
\end{equation}
where $G$ was defined in \eqref{G}. Using now the definition of  $R_N$ from \eqref{Fournier}, we deduce that for every $r >  0$
\[
\sum_{n=1}^{N}{\binom{N-1}{n-1} |\widehat{\psi}([n])| \big(r R_{N}\big)^{n-1}} \leq \binom{N-1}{\frac{N - \alpha -1}{2}} \frac{ G(r)}{2^{N-1}} \leq \sum_{n=1}^{N}{\binom{N-1}{n-1} |\widehat{\psi}([n])| r^{n-1}}.
\]
Integrating the previous inequality over the interval $[0, R]$ with $R>0$, yields that
\begin{equation}
\label{equa:BestEstimationRadiusSignFunctions3}
\frac{1}{N R_N} \sum_{n=1}^{N}{\binom{N}{n} |\widehat{\psi}([n])| \big(R R_N\big)^n} \leq \binom{N-1}{\frac{N - \alpha -1}{2}}
\frac{ \mathcal{I}(R) }{2^{N-1}} \leq \frac{1}{N} \sum_{n=1}^{N}{\binom{N}{n} |\widehat{\psi}([n])| R^{n}}\,.
\end{equation}
Replacing in the previous expression the variable $R$ first by  $\rho/R_N$ and later by $\rho$, we obtain  that
\begin{equation}\label{equa:BestEstimationRadiusSignFunctions666}
\frac{N}{2^{N-1}} \binom{N-1}{\frac{N-\alpha-1}{2}}  \mathcal{I}(\rho) \leq \sum_{n=1}^{N}{ \binom{N}{n} \, |\widehat{\psi}([n])| \rho^{n}} \leq \frac{N}{2^{N-1}} \binom{N-1}{\frac{N-\alpha-1}{2}} R_N \mathcal{I}\Big(\frac{\rho}{R_N}\Big)\,.
\end{equation}

Note that by the very definition of the Boolean radius, the number $\rho$ satisfies that
\begin{equation}\label{equa:BestEstimationRadiusSignFunctions4}
\sum_{n=1}^{N}{\binom{N}{n} |\widehat{\psi}([n])| \rho^{n}} = 1 - |\widehat{\psi}(\emptyset)|.
\end{equation}
The expectation of $\psi$, this is $\widehat{\psi}(\emptyset)$, can be easily calculated as
\begin{align*}
\widehat{\psi}(\emptyset) & = \sigma(x_{1} + \ldots +x_{N} > \alpha) - \sigma(x_{1} + \ldots +x_{N} < \alpha)\\
& = 2 \, \sigma(x_{1} + \ldots + x_{N} > \alpha) - 1,
\end{align*}
so \eqref{equa:BestEstimationRadiusSignFunctions4} yields that
\begin{equation}\label{equa:BestEstimationRadiusSignFunctions5}
\sum_{n=1}^{N}{\binom{N}{n} |\widehat{\psi}([n])| \rho^{n}} = 2 \, \sigma(x_{1} + \ldots + x_{N} > \alpha).
\end{equation}
Note that $x_{1} + \ldots +x_{N}$ depends just on the number of $1$'s and $-1$'s: if $m$ variables are equal to $-1$ and $N-m$ are equal to $1$ then $x_{1} + \ldots +x_{N} = N - 2m$, so that
\[
\sigma(x_{1} + \ldots + x_{N} > \alpha) = \frac{1}{2^{N}} \sum_{N - 2m > \alpha}{ \binom{N}{m} }\,,
\]
and hence \eqref{equa:BestEstimationRadiusSignFunctions5} leads to
\begin{equation*}
\sum_{n=1}^{N}{\binom{N}{n} |\widehat{\psi}([n])| \rho^{n}} =  \frac{1}{2^{N-1}} \sum_{m \leq \frac{N - \alpha -1}{2}}{\binom{N}{m}}.
\end{equation*}
Applying this equality to the middle sum of \eqref{equa:BestEstimationRadiusSignFunctions666}, we conclude that both inequalities of \eqref{equa:BestEstimationRadiusSignFunctions1} are valid.
\end{proof}

An immediate consequence of  Theorem \ref{Theo:BestEstimationRadiusSignFunctions} is the following corollary.

\begin{Coro}\label{Coro:BestEstimationRadiusSignFunctions}
For $N \in \N$ and $0 \leq \alpha < N$ such that $N-\alpha$ is odd let $\rho = \rho_{\alpha, N}$ be the Boolean radius of
$\psi = \psi_{\alpha, N}$. Then \begin{equation}\label{equa:BestEstimationRadiusSignFunctions1}
\mathcal{I}(\rho) \leq \frac{1}{N  \binom{N-1}{\frac{N - \alpha - 1}{2}}} \sum_{n \leq \frac{N - \alpha -1}{2}}{\binom{N}{n}}
\leq \frac{1}{3}\,\mathcal{I}(3\rho)\,.
\end{equation}
\end{Coro}

\begin{proof} Since $\frac{1}{3} \leq R_N$ and $G$ is a strictly increasing continuous function on $\mathbb{R}_{+}$,
the function $r \mapsto \mathcal{I}(r)/r$ is increasing on $(0, \infty)$. In consequence
\[
R_N \mathcal{I}\Big(\frac{\rho}{R_N}\Big) \leq \frac{1}{3}\,\mathcal{I}(3\rho)\,,
\]
and so Theorem \ref{Theo:BestEstimationRadiusSignFunctions} applies.
\end{proof}

We now start to analyze \eqref{equa:BestEstimationRadiusSignFunctions1}. In order to control the  combinatorial numbers appearing in  \eqref{equa:BestEstimationRadiusSignFunctions1}
we need an auxiliary result from \cite{Mckay}, and to state this  result  we introduce the function $Y\colon [0,\infty) \to [0,\infty)$ defined by
\begin{equation} \label{definition}
Y(x) = e^{\frac{1}{2}x^2} \int_{x}^{\infty}e^{-\frac{1}{2}t^2}\,dt, \quad\, x\geq 0.
\end{equation}

\begin{Lemm}\label{Lemm:McKayEstimations}
For each $N \in \N$ and each $0 \leq \alpha <N$ with $N - \alpha$ being an odd integer, there exists a~real number $0 \leq c_{\alpha,N} \leq \sqrt{\pi/2}$ such that
\[
\sum_{n \leq \frac{N - \alpha - 1}{2}}{\binom{N}{n}} = \sqrt{N}\,\binom{N-1}{\frac{N - \alpha -1}{2}}\,Y\left( \frac{\alpha + 1}{\sqrt{N}} \right)
\,\exp{\left(  \frac{c_{\alpha, N}}{\sqrt{N}} \right)}\,.
\]
\end{Lemm}

\begin{proof}
It is proved in \cite{Mckay}  that given $N \geq 1$ and $N/2 \leq k \leq N$, we have that
\begin{equation}\label{equa:MckayEstimation}
\sum_{j=k}^{N}{\binom{N}{j}} = \sqrt{N} \, \binom{N-1}{N - k} \, Y\left( \frac{2k - N}{\sqrt{N}} \right) \,  \exp{\left( \frac{E(k,N)}{ \sqrt{N}} \right)}
\end{equation}
for some real number
\[
0 \leq E(k,N) \leq \min{\left(\sqrt{\frac{\pi}{2}}, \frac{2 \sqrt{N}} {2k - N}\right)}.
\]
The properties of the binomial coefficients yield that
\[
\sum_{n \leq \frac{N - \alpha - 1}{2}}{\binom{N}{n}} =  \sum_{n \leq \frac{N - \alpha - 1}{2}}{\binom{N}{N-n}} = \sum_{n = \frac{N + \alpha +1}{2}}^{N}{\binom{N}{n}}.
\]
Therefore, replacing $k=\frac{N + \alpha + 1}{2}$ in \eqref{equa:MckayEstimation} we get the result.
\end{proof}

We finally need another technical result which sheds more light on the functions $G$ and $\mathcal{I}$ which were defined
in \eqref{G} and \eqref{I}, and appeared in  Lemma \ref{Theo:BestEstimationRadiusSignFunctions}.

\begin{Lemm}\label{Lemm:supremumTorusRewritten}
For $0 \leq \alpha \leq N-1$ we have
\[
G(r) = \begin{cases}
(1 + r)^{\frac{N+\alpha-1}{2}} (1 - r)^{\frac{N - \alpha - 1}{2}} &,\,\ r \in [0, r_{\alpha}], \\
(1 + r^{2})^{\frac{N-1}{2}} \left( 1 + \frac{\alpha}{N-1} \right)^{\frac{N+\alpha - 1}{4}}
\left( 1 - \frac{\alpha}{N-1} \right)^{\frac{N-\alpha-1}{4}} &,
\,\, r\in [r_{\alpha}, 1]
\end{cases}
\]
where
\[
r_{\alpha}:= \frac{\alpha}{(N-1) + \sqrt{(N-1)^{2} - \alpha^{2}}}\,.
\]
\end{Lemm}

\begin{proof}
Putting $z = c + it \in \mathbb{T}$ with $c^{2} + t^{2} = 1$ we can write
\[
G(r) = \sup_{t \in [-1,1]}{(1 + r^{2} + 2 r t)^{\frac{N+\alpha - 1}{4}} (1 + r^{2} - 2 r t)^{\frac{N-\alpha - 1}{4}}}, \quad\, r\in [0,1].
\]
To show the required formula for $G$ we define for a given $r\in [0, 1]$ the function $g_r\colon [-1, 1] \to \mathbb{R}_+$ by
\begin{align}
\label{equa:functionAux1}
g_r(t)= (1 + r^{2} + 2 r t)^{\frac{N+\alpha - 1}{4}} (1 + r^{2} - 2 r t)^{\frac{N-\alpha - 1}{4}}, \quad\, t\in [-1, 1].
\end{align}
Calculus yields that the first derivative
\[
g_{r}'(t) = r^{2} (N-1) (1+r^{2} + 2rt)^{\frac{N+\alpha-1}{4}-1} (1 + r^{2} - 2rt)^{\frac{N-\alpha-1}{4}-1}
\left( t_{\alpha, r} - t \right), \quad\, t\in (-1, 1)
\]
where
\[
t_{\alpha, r} := \frac{\alpha (1 + r^{2})}{2r(N-1)}.
\]
We now observe that $t_{\alpha, r} \leq 1$ if and only if $r\in [r_{\alpha}, 1]$, and that $g_r$ attains the maximum
at $t_{\alpha, r}$. Thus
\[
G(r) = g_r(t_{\alpha, r}) = (1 +r^{2})^{\frac{N-1}{2}} \left( 1 + \frac{\alpha}{N-1} \right)^{\frac{N+\alpha - 1}{4}} \left( 1 - \frac{\alpha}{N-1} \right)^{\frac{N-\alpha-1}{4}}.
\]
In the case when $r\in [0, r_{\alpha}]$, then $g$ is increasing on $[-1,1]$ and so
\[
G(r) = g(1) = (1 + r)^{\frac{N+\alpha-1}{2}} (1 - r)^{\frac{N-\alpha-1}{2}},
\]
and this completes the proof.
\end{proof}

Now we combine the preceding two lemmas to exploit the implicit estimate for $\rho(\psi_{\alpha, N})$ given in  Theorem
 \ref{Theo:BestEstimationRadiusSignFunctions}.

\begin{Prop}
\label{Theo:YEstimationRadiusSignFunctions}
There exists an absolute constant $C>0$ such that for each $\alpha, N \in \N$ with $0 \leq \alpha < N$ and $N - \alpha$ is an odd integer we have
\[
C e^{\frac{2}{\sqrt{N}}} \frac{1}{\sqrt{N}}\,Y\left(\frac{\alpha +1}{\sqrt{N}} \right) \leq \rho(\psi_{\alpha, N}) \leq e^{\frac{2}{\sqrt{N}}}
\frac{1}{\sqrt{N}}\,Y\left(\frac{\alpha +1}{\sqrt{N}} \right).
\]
\end{Prop}

\begin{proof}
We are going to apply  Lemma \ref{Theo:BestEstimationRadiusSignFunctions}. If we denote $\rho = \rho(\psi_{\alpha, N})$ to simplify, then combining \eqref{equa:BestEstimationRadiusSignFunctions1} and Lemma \ref{Lemm:McKayEstimations} we obtain that
\begin{equation}\label{equa:BooleanRadiusSignFunctionAux1}
e^{\frac{-2}{\sqrt{N}}} \, \mathcal{I}(\rho)  \leq  \frac{1}{\sqrt{N}} Y\left( \frac{\alpha +1}{\sqrt{N}} \right)
\leq \frac{1}{3}\,\mathcal{I}(3\rho).
\end{equation}
Now observe that by its definition given in  \eqref{definition}, the $Y$ is a~decreasing function on $\mathbb{R}_{+}$. Since the function $G$ defined in \eqref{G}  satisfies $G(r) \geq 1$ for any $r \geq 0$ (replace just $z$ by $i$ in the expression inside the supremum on the definition of $G(r)$), this yields  the right-hand of the required
estimate:
\begin{align}
\label{equa:upperBoundBooleanSign}
\rho \leq \int_{0}^{\rho} G(r)\:dr & = \mathcal{I}(\rho) \leq e^{\frac{2}{\sqrt{N}}}
\frac{1}{\sqrt{N}} Y\left(\frac{\alpha +1}{\sqrt{N}} \right).
\end{align}
To get the  lower bound for $\rho$, note that since $G$ is an increasing function  on [0, 1] (this can be deduced
immediately by the maximum modulus theorem) we have that
\begin{equation}
\label{equa:lowerBoundBooleanSign}
\frac{1}{3}\,\mathcal{I}(3\rho) = \frac{1}{3}\int_{0}^{3 \rho}{G(r) \: dr} \leq \rho\, G(3 \rho).
\end{equation}
We are going to bound now  $G(3\rho)$  by a constant independent of $N$ and $\alpha$. For that we make use
of Lemma \ref{Lemm:supremumTorusRewritten}, where an explicit expression for $G$ is presented.
Indeed, we are  going to distinguish two cases for  $r\in [0, 1]$ according to whether
$3 \rho \leq r_{\alpha}$ or $3 \rho \geq r_{\alpha}$. Assuming that $3 \rho \leq r_{\alpha}$, we apply
\eqref{equa:upperBoundBooleanSign}  to deduce that
\begin{align*}
G(3\rho) & = (1 + 3\rho)^{\alpha} \, (1 - 9 \rho^{2})^{\frac{N-\alpha-1}{2}} \leq (1 + 3 \rho)^{\alpha} \leq \exp{(3 \rho \alpha)}.
\end{align*}
The above formula for $Y$ shows that $Y(x) \leq 1/x$ for all $x>0$ and so it follows from \eqref{equa:upperBoundBooleanSign} that
\[
\rho \leq e^2 \frac{1}{\sqrt{N}}\,Y\left( \frac{\alpha +1}{\sqrt{N}} \right) \leq \frac{e^2}{\alpha +1}.
\]
This leads to $G(3\rho) \leq \exp(3e^2)$ in the case when $3\rho \leq r_{\alpha}$.

Consider now the case $3 \rho \geq r_{\alpha}$. Since $Y$ is a decreasing function on $[0,\infty)$, we conclude that
\begin{align*}
\rho & \leq \frac{e^2}{\sqrt{N}}\,Y\left( \frac{\alpha +1}{\sqrt{N}} \right) \leq
\frac{e^2}{\sqrt{N}}\,Y(0) =  \frac{C_0}{\sqrt{N}},
\end{align*}
where $C_0 = e^{2} \sqrt{\frac{\pi}{2}}$.

Combining the formula for $r_{\alpha}:= \frac{\alpha}{(N-1) + \sqrt{(N-1)^{2} - \alpha^{2}}}$ with
\eqref{equa:upperBoundBooleanSign} we get that
\[
\alpha \leq 2(N-1) r_{\alpha} \leq 6\rho \sqrt{N-1} \leq 6 C_{0}.
\]
Applying this last estimate and \eqref{equa:upperBoundBooleanSign}  we conclude that
\begin{align*}
G(3 \rho) & \leq \left(1 + \frac{\alpha}{N-1} \right)^{\frac{\alpha}{2}} (1 + 9\rho^{2})^{\frac{N-1}{2}}
\leq \exp{\left( \frac{\alpha^{2}}{2(N-1)} + 9 \rho^{2} \frac{N-1}{2} \right)}\\
& \leq \exp{\left( \frac{36 C_{0}^2 }{2} + \frac{9 C_{0}^{2}}{2} \right)} \leq \exp{(23 C_{0}^{2})}.
\end{align*}
Combining \eqref{equa:BestEstimationRadiusSignFunctions1}, \eqref{equa:lowerBoundBooleanSign} and \eqref{equa:upperBoundBooleanSign}
yields the desired estimations.
\end{proof}

Let us focus now on the proof of  Theorem \ref{Theo:BooleanRadiusSignFunction}. Using again the fact that $Y$ is a~decreasing
function with $Y(x) \leq 1/x$ on $\mathbb{R}_{+}$, it can be easily shown (see \cite{Mckay}) that
\begin{equation}\label{equa:EstimatingY}
\frac{x}{1 + x^{2}} \leq Y(x) \leq \frac{1}{x} \,\quad \mbox{for every $x>0$}.
\end{equation}
Note that these estimations are not accurate when $x$ goes to zero, as $Y(0)$ is a~well-defined positive real number.

\vspace{2mm}

\begin{proof}[Proof of Theorem \ref{Theo:BooleanRadiusSignFunction}]
Let $N \in \N$ and $0 \leq \alpha < N$. We start with the following observations: using the monotonicity of $Y$ we have
for $\alpha \leq \sqrt{N}$ the bounds
\begin{equation*}
\frac{Y(2)}{\sqrt{N}} \leq \frac{1}{\sqrt{N}}Y\Big(\frac{\alpha + 1}{\sqrt{N}}\Big) \leq  \frac{Y(0)}{\sqrt{N}};
\end{equation*}
while for $\alpha \geq \sqrt{N}$ we can use the estimations \eqref{equa:EstimatingY} to deduce that
\begin{equation*}
\frac{1}{\sqrt{N} + \alpha + 1} \leq \frac{\alpha + 1}{N + (\alpha + 1)^{2}} \leq \frac{1}{\sqrt{N}}Y\Big(\frac{\alpha + 1}{\sqrt{N}}\Big)
\leq \frac{1}{\alpha+1} \leq \frac{2}{\alpha + \sqrt{N}}
\end{equation*}
We can summarize the above inequalities by saying that there exists an absolute constant $C_{0} > 0$ (independent of $\alpha$ and $N$)
which satisfies
\begin{equation}
\frac{C_{0}^{-1}}{\alpha + \sqrt{N}} \leq \frac{1}{\sqrt{N}}Y\Big(\frac{\alpha + 1}{\sqrt{N}}\Big) \leq  \frac{C_{0}}{\alpha + \sqrt{N}}.
\end{equation}
We  divide now the proof of the theorem into two cases.

\noindent
\textbf{Case 1:} Assume that $N-\alpha$ is an odd natural number. Then by \eqref{equa:EstimatingY} and  Proposition \ref{Theo:YEstimationRadiusSignFunctions} there is
 an absolute  constant $C_0>0$ such that
\begin{equation}
\label{equa:boundBooleanSignRestricted}
C_{0}^{-1} \frac{1}{\alpha + \sqrt{N}} \leq \rho(\psi_{\alpha, N}) \leq  C_{0} \, \frac{1}{\alpha + \sqrt{N}}.
\end{equation}

\noindent
\textbf{Case 2:} Without any restriction on $0\leq \alpha < N$, let $0 \leq n <N/2$ be an integer such that
$\alpha \in [N-2n-2, N-2n)$. Note that for each $x \in \{ \pm 1 \}^{N}$ the sum $x_{1} + \ldots + x_{N}$ is
equal to $N-2k$ for some integer $k$. Thus, if $\alpha' \in [N-2n-2, N-2n)$ then $\psi_{N, \alpha} = \psi_{N, \alpha'}$.
In particular this is true for $\alpha' = N-2n-1$, which satisfies that $N - \alpha'$ is an odd natural number.
We can apply then \eqref{equa:boundBooleanSignRestricted} to obtain that
\[
C_{0}^{-1} \, \frac{1}{\alpha' + \sqrt{N}} \leq \rho(\psi_{\alpha, N}) = \rho(\psi_{\alpha', N})
\leq  C_{0} \, \frac{1}{\alpha' + \sqrt{N}}.
\]
Finally, using that $|\alpha - \alpha'| \leq 1$ we conclude the result.
\end{proof}

The argument in Proposition \ref{Theo:BestEstimationRadiusSignFunctions} can be refined for the Majority function
$\Maj_{N}$ ($N$ odd) to give the precise estimation of its Boolean radius as $N$ goes to infinity.

\begin{Coro}\label{Coro:BooleanRadiusMaj}
The Boolean radius of the Majority function satisfies
\[
\rho\big( \Maj_{N} \big) = \frac{\gamma}{\sqrt{N}} (1 + o(1)),
\]
where $\gamma > 0$ is the only real number satisfying $\int_{0}^{\gamma}{e^{\frac{u^{2}}{2}} \: du} = \sqrt{\frac{\pi}{2}}$.
\end{Coro}

\begin{proof}
In the case of the Majority function $\Maj_{N} = \psi_{N,0}$ where $N$ is odd, we have that replacing $z$ by $ri$ in \eqref{equa:LemmFourierSpectrumSignFunctionsAux1}
and comparing the coefficients of $r^{m}$ $(m \in \N)$ on both sides we get that
\[ \sum_{n=1}^{N}{ \binom{N-1}{n-1} |\widehat{\psi}_{N}([n])| r^{n-1} } = \binom{N-1}{\frac{N-1}{2}}
\frac{1}{2^{N-1}} (1 + r^{2})^{\frac{N-1}{2}}. \]
Denote by $\rho_{N}$ the Boolean radius of $\psi_{N}$. Then, integrating between $0$ and $\rho_{N}$ in the
previous expression and using that $\widehat{\psi}_{N}(\emptyset) = 0$ we deduce that
\[
\frac{1}{N}=\frac{1}{N} \sum_{n=1}^{N}{ \binom{N}{n} |\widehat{\psi}([n])| \rho_{N}^{n} }
= \binom{N-1}{\frac{N-1}{2}} \frac{1}{2^{N-1}} \int_{0}^{\rho_{N}}{(1 + r^{2})^{\frac{N-1}{2}} \: dr},
\]
where the first equality is just the definition of Boolean radius. Using Stirling's approximation formula and the change of variables $r=u/\sqrt{N}$, we obtain that if
$C_{N} := \rho_{N} \sqrt{N}/\gamma$ (where $\gamma$ is as above), then
\[
\int_{0}^{C_{N} \gamma}{\Big( 1 + \frac{u^{2}}{N} \Big)^{\frac{N-1}{2}} \: d u} =
\frac{1}{\sqrt{N}} \frac{2^{N-1}}{\binom{N-1}{\frac{N-1}{2}}} = \sqrt{\frac{\pi}{2}} \big( 1 + o(1)\big).
\]
Hence $C_{N}$ is bounded. Using the inequalities $(1 + \frac{1}{x})^{x} \leq e \leq (1 + \frac{1}{x})^{x+1}$
valid for $x>0$ we obtain that
\[
\int_{0}^{C_{N} \gamma}{e^{\frac{u^{2}}{2}} \: d u} = \sqrt{\frac{\pi}{2}} \big( 1 + o(1)\big).
\]
By the definition of $\gamma$ it follows that
\[
\int_{\gamma}^{C_{N} \gamma}{e^{\frac{u^{2}}{2}} \: du} = o(1).
\]
The mean value theorem yields that $C_{N} = 1 + o(1)$ as we wanted to prove.
\end{proof}

\subsection*{Upper bound}
Finally, we come back to the proof of  the upper estimate in Theorem  \ref{Theo:BooleanRadiusUnbiased}. We need  another  auxiliary
\begin{Lemm}\label{Lemm:EstimationSignAux} For every $0 \leq \alpha \leq N$ we have that
\begin{equation}\label{equa:EstimationSignAux}
\sigma(x_{1} + \ldots + x_{N} > \alpha) \geq  \exp{\left[-6 \Big(1 + \frac{\alpha}{\sqrt{N}}\Big)^2\right]}\,.
\end{equation}
\end{Lemm}
\begin{proof}
We are going to distinguish two cases. If $\alpha \geq N/2$ then we have the trivial bound
\[ \sigma(x_{1} + \ldots + x_{N} > \alpha) \geq 2^{-N} \geq \exp{(-N)} \geq  \exp{\left[-6 \Big(1 + \frac{\alpha}{\sqrt{N}}\Big)^2\right]}. \]
Assume that $\alpha < N/2$. By Lemma \ref{Lemm:McKayEstimations} we have that
\[\sigma(x_{1} + \ldots + x_{N} > \alpha) = \frac{1}{2^N} \sum_{n \geq \frac{N-\alpha - 1}{2}}{\binom{N}{n}} \geq \frac{\sqrt{N}}{2^N} \, \binom{N-1}{\frac{N-\alpha-1}{2}} \, Y \Big( \frac{\alpha +1}{\sqrt{N}} \Big).  \]
To bound the combinatorial number in the last expression let us recall some established bounds for binomial coefficients. From   \cite[p. 58]{Lovasz} we know that for each $0 \leq m \leq n$
\[
\binom{2n}{n-m} \geq \binom{2n}{n} \, \exp{\Big(\frac{-m^{2}}{n-m+1}\Big)} \geq \frac{2^{2n-1}}{\sqrt{n}}  \, \exp{\Big(\frac{-m^{2}}{n-m+1}\Big)}\,,
\]
where the last inequality is consequence of classical Stirling estimations. This yields that if $N$ is an odd number and $\alpha$ is as above, then
\[ \binom{N-1}{\frac{N- \alpha -1}{2}} \geq \frac{\sqrt{2}}{4} \, \frac{2^{N}}{\sqrt{N}} \, \exp{\Big(\frac{-\alpha^2}{2(N - \alpha + 1)}\Big)} \geq \frac{\sqrt{2}}{4} \, \frac{2^{N}}{\sqrt{N}} \exp{\Big( -\frac{\alpha^2}{N}\,, \Big)}  \]
where the last inequality follows from the hypothesis $\alpha \leq N/2$. If $N$ is even, then we can use the standard relations between binomial coefficients and the previous bound to get
\[ \binom{N-1}{\frac{N- \alpha -1}{2}} \geq \binom{N-2}{\frac{N- \alpha -1}{2}} \geq \frac{\sqrt{2}}{4} \, \frac{2^{N-1}}{\sqrt{N-1}} \exp{\Big( -\frac{(\alpha-1)^2}{N-1} \Big)} \,.\]
To avoid this distinction between $N$ odd or even, we can write the common lower bound
\[ \binom{N-1}{\frac{N- \alpha -1}{2}} \geq \frac{\sqrt{2}}{8} \, \frac{2^{N}}{\sqrt{N}} \, \exp{\Big( - \frac{2 \alpha^{2}}{N} \Big)}. \]
Thus, we have that
\begin{equation}\label{equa:estimationSignAux2}
\sigma(x_{1} + \ldots + x_{n} > \alpha)
\geq \frac{\sqrt{2}}{8} \, \exp{\Big( -\frac{2 \alpha^2}{N} \Big)}  \, Y\Big( \frac{\alpha +1}{\sqrt{N}} \Big)\,.
\end{equation}
To bound the last factor we simply turn to the original definition \eqref{definition} of the function $Y$ from which we easily obtain that
\begin{align*}
Y\Big( \frac{\alpha +1}{\sqrt{N}} \Big) & = \exp{\Big( \frac{(\alpha+1)^2}{2 N} \Big)} \int_{\frac{\alpha + 1}{\sqrt{N}}}^{  \infty}{e^{-\frac{t^2}{2}} \: dt}\\
& \geq \exp{\Big( \frac{\alpha^2}{2N} \Big)} \cdot \int_{\frac{\alpha + 1}{\sqrt{N}}}^{\frac{\alpha + 1}{\sqrt{N}} + 1}{e^{-\frac{t^2}{2}} \: dt} \\
& \geq \exp{\Big( \frac{\alpha^2}{2N} \Big)}  \exp{\Big[-\frac{1}{2}\Big(1+ \frac{\alpha+1}{\sqrt{N}} \Big)^2\Big]}\\
& \geq  \exp{\Big( \frac{\alpha^2}{2N} \Big)}  \exp{\Big[-2\Big(1+ \frac{\alpha}{\sqrt{N}} \Big)^2\Big]}\,.
\end{align*}
Applying this last estimation to \eqref{equa:estimationSignAux2} we conclude that
\begin{align*}
\sigma(x_{1} + \ldots + x_{N} > \alpha) & \geq \frac{\sqrt{2}}{8} \exp{\Big( - \frac{3\alpha^{2}}{2N} \Big)} \, \exp{\Big[-2\Big(1+ \frac{\alpha}{\sqrt{N}} \Big)^2\Big]}\\
& \geq  \exp{\Big[-6\Big(1+ \frac{\alpha}{\sqrt{N}} \Big)^2\Big]} \,.
\end{align*}
This completes the proof.
\end{proof}

\begin{proof}[Proof of Theorem \ref{Theo:BooleanRadiusUnbiased}]
Let $N \in \N$ and $\frac{1}{2^{N}} \leq \delta \leq 1$. It has already been proved in Theorem \ref{Theo:biasedBooleanRadius} that
\[
\rho(\mathcal{B}^{N}_{\delta}) \geq \frac{1}{5 \sqrt{N} \sqrt{\log{(2/\delta)}}}.
\]
To prove the upper bound, we will make use of the threshold functions $\psi_{\alpha, N},\, 0 \leq \alpha < N$ which were studied in the previous section.  Recall that
\[ \big| \mathbb{E}[\psi_{\alpha, N}] \big| = 1 - 2 \sigma(x_{1} + \ldots + x_{N} > \alpha) = \big[ 1 - 2 \sigma(x_{1} + \ldots + x_{N} > \alpha) \big] \, \| \psi_{\alpha, N}\|_{\infty}. \]
Then, by Lemma \ref{Lemm:EstimationSignAux} we deduce that $\psi_{\alpha, N} \in \mathcal{B}_{\delta}^{N}$ whenever
\begin{equation}\label{equa:BooleanRadiusUnbiasedAux1}
\left(1 + \frac{\alpha}{\sqrt{N}} \right)^{2} = \frac{\log{(2/\delta)} }{6}.
\end{equation}
If $\delta < 2/e^{6}$, then we can find $0 \leq \alpha < N$ satisfying  \eqref{equa:BooleanRadiusUnbiasedAux1}, so that
\[ \rho(\mathcal{B}^{N}_{\delta}) \leq \rho(\psi_{N, \alpha}) \leq \frac{C}{\alpha + \sqrt{N}} = \frac{C}{\sqrt{N} \Big(1 + \frac{\alpha}{\sqrt{N}}\Big)} = \frac{C \sqrt{6}}{\sqrt{N} \sqrt{\log{(2/\delta)}}}. \]
If $\delta \geq 2/e^{6}$, then, since $\psi_{N,0} \in \mathcal{B}^{N}_{\delta}$ holds trivially, we immediately have that
\[ \rho(\mathcal{B}_{\delta}^{N}) \leq \rho(\psi_{N,0}) \leq O\Big( \frac{1}{\sqrt{N}} \Big) \,. \qedhere\]
\end{proof}


\end{document}